\documentclass{amsart}
\usepackage{amssymb}
\usepackage{amsbsy}
\usepackage{amscd}
\usepackage{amsmath}
\usepackage{amsthm}
\usepackage{amsxtra}
\usepackage{latexsym}
\usepackage{tikz}
\usepackage{mathdots}
\usepackage{mathtools}
\usepackage[new]{old-arrows}

\usepackage{pict2e}

\usepackage{caption}

\usepackage[mathscr]{eucal}
\usepackage{mathrsfs}
\usepackage{upgreek}
\usepackage{wasysym}
\usepackage[all,cmtip,dvipdfmx]{xy}
\usepackage{xcolor}
\definecolor{rouge}{rgb}{0.85,0.1,.4}
\definecolor{bleu}{rgb}{0.1,0.2,0.9}
\definecolor{violet}{rgb}{0.7,0,0.8}

\usepackage[
colorlinks=true,linkcolor=bleu,urlcolor=violet,citecolor=rouge]{hyperref}

\newcommand{\V}{\mathbb{V}}
\renewcommand{\L}{\mathbb{L}}

\newcommand{\W}{\mathscr{W}}

\newcommand{\bra}{{\langle}}
\newcommand{\ket}{{\rangle}}

\newcommand{\Lam}{\Lambda}

\newcommand{\cprime}{$'$}
\newcommand{\on}{\operatorname}
\newcommand{\+}{\mathop{\oplus}}

\newcommand{\mc}{\mathcal}
\newcommand{\mf}{\mathfrak}

\newcommand{\fing}{\mf{g}}

\newcommand{\affg}{\widehat{\mf{g}}}

\newcommand{\affh}{\widehat{\mf{h}}}

\newcommand{\Z}{\mathbb{Z}}
\newcommand{\C}{\mathbb{C}}

\newcommand{\ra}{\rightarrow}
\newcommand{\lam}{\lambda}

\newcommand{\vac}{{|0\rangle}}

\newcommand{\bs}{\boldsymbol}

\def\g{\mathfrak{g}}

\def\h{\mathfrak{h}}
\def\n{\mathfrak{n}}

\def\O{\mathbb{O}}

\def\P{\mathscr{P}}
\def\sl{\mathfrak{sl}}
\def\so{\mathfrak{so}}
\def\sp{\mathfrak{sp}}

\def\d{d}

\def\bs{\boldsymbol}
\def\P{\mathscr{P}}

\def\le{\leqslant}

\def\ge{\geqslant}
\def\geq{\geqslant}

\DeclareMathOperator{\ad}{ad}

\newcommand{\DeltaLong}{\Delta_{\text{long}}}
\newcommand{\DeltaShort}{\Delta_{\text{short}}}
\newcommand{\Deltahatre}{\widehat{\Delta}^{\text{re}}}
\DeclareMathOperator{\Adm}{Adm}
\newcommand{\thetashort}{\theta_{\text{s}}}

\theoremstyle{theorem}
\newtheorem{Th}{Theorem}[section]

\newtheorem{Pro}[Th]{Proposition}

\newtheorem{lemma}[Th]{Lemma}

\theoremstyle{remark}

\newlength{\larg}
\title[Nilpotent orbits arising from admissible representations]
{On the nilpotent orbits arising from admissible affine vertex algebras}

\author[Tomoyuki Arakawa]{Tomoyuki Arakawa\textsuperscript{1}}
\address{\textsuperscript{1}Research Institute for Mathematical Sciences
\\ Kyoto University\\ 
 Kyoto 606-8502 JAPAN}
\email{arakawa@kurims.kyoto-u.ac.jp}

\author[Jethro van Ekeren]{Jethro van Ekeren\textsuperscript{2}}
\address{\textsuperscript{2}
Departamento de Matem\'{a}tica Aplicada\\
Instituto de Matem\'{a}tica e Estat\'{i}stica\\
Universidade Federal Fluminense\\
Niter\'{o}i, RJ 24.210-201 BRAZIL}
\email{jethrovanekeren@gmail.com}

\author[Anne Moreau]{Anne Moreau\textsuperscript{3}}
\address{\textsuperscript{3}Facult\'{e} des Sciences d'Orsay\\
Universit\'{e} Paris-Saclay\\
91405 Orsay FRANCE}
\email{anne.moreau@universite-paris-saclay.fr }

\begin{document}

\maketitle

\begin{abstract}
We give a simple description of the closure of the nilpotent orbits appearing as associated varieties of 
admissible affine vertex algebras in terms of primitive ideals.
\end{abstract}
\section{Introduction}

The notion of associated variety of a primitive ideal in a universal enveloping algebra  is a widely used and very useful one in representation theory, permitting the application of geometric methods, see e.g., \cite{Vog91}. Recently, a similar notion has been introduced for vertex algebras 
(\cite{Ara12}) which has turned out to be useful not only in the representation theory of vertex algebras (e.g.,\cite{A2012Dec,AM15,AraKaw18,EkeHel,AEkeren19})
but also in connection with {four dimensional $\mc{N}=2$ superconformal field theories} via the 4D/2D correspondence discovered in \cite{BeeLemLie15},
see e.g., \cite{SonXieYan17, Son17, BeeRas,AraMT,BonMenRas19,PanPee20,Ded,XieYan2019lisse}.

Let $\g$ be a simple complex Lie algebra and $G$ the corresponding adjoint algebraic group. For any choice of level $k$ the associated variety $X_{\L_k(\g)}$ of the simple affine vertex algebra $\L_k(\g)$ is a conic, $G$-invariant 
subvariety of $\g^*$.
While the associated variety of a primitive ideal in $U(\g)$ is always contained in the nilpotent cone $\mc{N} \subset \g^*$, this is not necessarily the case for the variety $X_{\L_k(\g)}$. 
Nevertheless, 
it was shown in \cite{Ara09b} that 
$X_{\L_k(\g)}$ is the closure of a nilpotent orbit $\O_q$, 
which depends only on the denominator $q\in \Z_{\geq 1}$ of the rational number $k$
whenever $\L_k(\g)$ is an admissible representation \cite{KacWak89} of the affine Kac-Moody algebra
$\widehat{\g}$.
More explicitly,
we have
\begin{align*}
\overline{\O}_q=
\begin{cases}
\mc{N}_q & \text{if }(q,\check{r})=1,\\
{}^L\mc{N}_{q/\check{r}} & \text{if }(q,\check{r})\ne 1,
\end{cases}
\end{align*}
where by definition
\begin{align}
\label{eq:Nq}
\begin{split}
\mc{N}_q &= \{x\in \g \colon (\ad x)^{2q}=0\}, \\
\text{and} \quad {}^L\mc{N}_q &= \{x\in \g \colon \pi_{\thetashort}(x)^{2q} = 0\},
\end{split}
\end{align}
and $\pi_{\theta_{\text{s}}}$ is the irreducible finite-dimensional representation of $\g$
with highest weight $\thetashort$,
$\thetashort$ is the highest short root of $\g$ and $\check{r}$ is the lacing number of $\g$, that is, the maximal number of edges in the Dynkin diagram of $\g$. The irreducibility of the varieties of \eqref{eq:Nq} was checked 
in \cite{Ara09b} so that the nilpotent orbit $\O_q$ is indeed well-defined. 
Prior to \cite{Ara09b},
the irreducibility of variety $\overline{\O}_q$ had been proved in \cite{Geo04} 
in most cases
and
the orbits $\O_q$ were studied in  \cite{ElaKacVin08}
as 
exceptional nilpotent orbits for those  
of principal Levi type and
$(q,\check{r})=1$.

In connection with the 4D/2D correspondence mentioned above,
the variety $\overline{\O}_q$ also appears as the Higgs branch of some 
{A}rgyres-{D}ouglas theories \cite{Dan,SonXieYan17,WanXie}.

The purpose of this note is to give a 
simple description of the  orbit $\O_q$ in terms of primitive ideals. It is well known that the associated variety of the annihilator of an integral highest weight representation of $\g$ is the closure of a {\emph{special} \cite{Lus79}} nilpotent orbit (\cite{BarVog82,BarVog83}). Since the nilpotent orbit $\O_q$ is not special in general, we need to consider non-integral highest weight representations.

Let $\rho$ be the half sum of positive roots of $\g$ and 
$\check{\rho}$ the half sum of positive coroots of $\g$. We denote by $J_{\lam}\subset U(\g)$ the annihilator of the simple highest weight representation $L(\lam)$
of $\g$ with highest weight $\lam$, and by $\on{Var}(I)$ the associated variety of a primitive ideal $I \subset U(\g)$. Our main result is the following.
\begin{Th}\label{MainThm}
Let $q$ be a positive  integer, and set
\begin{align*}
\lam_q=\begin{cases}
\rho/q-\rho&\text{if }(q,\check{r})=1,\\
\check{\rho}/q-\rho&\text{if }(q,\check{r})\ne 1.
\end{cases}
\end{align*}
Then  $$\on{Var}(J_{ \lam_q})=\overline{\O}_q.$$
\end{Th}
The proof of Theorem \ref{MainThm} 
goes as follows.
We 
use the representation theory of admissible {affine vertex algebras}
to show that   $\on{Var}(J_{\lam_q})$
is contained in $\overline{\O}_q$.
Then we show that 
the dimension of $\on{Var}(J_{ \lam_q})$
is equal to 
that of $\overline{\O}_q$ using the formulas obtained in \cite{ElaKacVin08}. Since both varieties are irreducible, the equality follows. 

Let $k = -\check{h} + p/q$ be an admissible level and 
$\L_k(\lambda)$ the irreducible representation of $\widehat{\g}$ of highest weight $\lambda + k\Lambda_0$.
By \cite{A12-2}, the irreducible $\widehat{\g}$-module
$\L_k(\lambda)$ is a module over the vertex algebra $\L_k(\g)$ if and only if 
$\lam+k\Lam_0$ is an admissible weight whose integral root system is isomorphic to that 
of $k\Lam_0$. 
This vertex algebra is quasi-lisse \cite{Ara09b} (see \cite{AraKaw18} for the definition), and it is known that the characters, that is to say graded dimensions, of ordinary modules of lisse, and more generally quasi-lisse, vertex algebras have modular invariance properties \cite{Zhu96,Miyamoto,AraKaw18}. This structure, in turn, permits the application of techniques from the theory of modular forms to the representation theory of such vertex algebras. In our recent work \cite{AvEM} we have applied such techniques to discover and prove isomorphisms between affine vertex algebras and affine $W$-algebras known as collapsing levels. In this context the minimal value among the conformal dimensions of ordinary modules of a quasi-lisse vertex algebra is an important invariant which controls asymptotic behaviour of the characters, and the results of the present paper relate to this circle of ideas through the following result: 
\begin{Pro}[{\cite[Proposition 3.12]{AvEM}}]
\label{prop:min_conformal_Lk}
Assume that $k$ is admissible and that 
$\lam+k\Lam_0$ is an admissible weight whose integral root system is isomorphic to that 
of $k\Lam_0$. Then 
\begin{align*}
\on{Var}(J_{\lam})\subset \overline{\mathbb{O}}_q,\quad 
| \lam+\rho|\geq | \lam_q+\rho|,
\end{align*}
and the equality holds if and only if 
$\on{Var}(J_{ \lam})= \overline{\mathbb{O}}_q$
and $ J_{\lam}= J_{\lam_q}$.
In particular, $\L_k(\lambda_q)$ is a simple representation of the vertex algebra $\L_k(\g)$ and 
possesses the minimal conformal dimension
among the simple $\L_k(\g)$-modules that belong to the category $\mc{O}_k$.
\end{Pro}
We remark that Theorem \ref{MainThm} and Proposition \ref{prop:min_conformal_Lk} together are somewhat similar in spirit to \cite[Proposition 5.10]{BarVog85}. 

One sees from the statement of Proposition \ref{prop:min_conformal_Lk} that the weight $\lam_q+k\Lam_0$ 
is not the unique 
one that gives the minimal conformal dimension unless $k \in \Z_{\ge 0}$. The statement is in some ways more pleasant, however, after Drinfeld-Sokolov reduction. 
Let $f\in \g$  be a nilpotent element, and let $H_{DS,f}^0(?)$ be the Drinfeld-Sokolov reduction 
functor associated with $f$. 
The following theorem is proven in \cite{AvEM}, too. 

\begin{Th}[{\cite[Theorem 4.11]{AvEM}}]
\label{Th:min-conf-dim}
Let $k$ be admissible,
and assume that $f$ admits an even good grading defined by a semisimple element $x^0$.
\begin{enumerate}
\item Let $f\in \overline{\mathbb{O}}_k$, so that $H_{DS,f}^0(\L_k(\g))\cong \W_k(\g,f)$ (see \cite{AEkeren19}). 
Then among the simple positive energy $\W_k(\g,f)$-modules, the minimal conformal dimension is given by $h_{\lam_q}$ where in general
$$h_\lam =\frac{|\lam+\rho|^2-|\rho|^2}{2(k+\check{h})}-\frac{k+\check{h}}{2}|x^0|^2+(x^0,\rho)$$ 
is the conformal dimension of the $\W_k(\g,f)$-module $H_{DS,f}^0(\L_k(\lam))$.

 \item Suppose further that  $f\in \mathbb{O}_k$, so that $\W_k(\g,f)$ is lisse. Then the minimal conformal dimension~$h_{\lam_q}$ is realised by a unique 
 simple $\W_k(\g,f)$-module. 
\end{enumerate}
\end{Th}

It seems to be widely believed that 
 a rational, lisse, simple, self-dual conformal vertex algebra $V$ 
 admits a a unique  simple module of minimal conformal dimension. 
Theorem \ref{Th:min-conf-dim} confirms 
this assertion   
for exceptional $W$-algebras (\cite{KacWak03,Ara09b})
(in a wider sense, see \cite{AEkeren19}) 
that are lisse \cite{Ara09b} and rational 
(\cite{Ara13, Ara09b, CreLin18,AEkeren19, Fas22,CreLin})\footnote{The rationality of
exceptional $W$-algebras has been recently proved in full generality  by McRae \cite{McRae21}.}.

In another direction, let us mention that
one can show that 
\begin{align}
\label{eq:small_quantum}
\overline{\mathbb{O}}_q\cong \on{Spec}H^\bullet(\mf{u}_{\zeta}(\g),\C)
\end{align}
based on 
the works \cite{GinKum93,BenNakPar14},
where $\mf{u}_{\zeta}(\g)$ is the small quantum group associated with $\g$ at 
the $q$-th root $\zeta$ of unity 
provided that $q$ is odd and not a bad prime for $\g$. 
Indeed, for $q > h$, with $h$ the Coxeter number of $\g$,
$\on{Spec}H^\bullet(\mf{u}_{\zeta}(\g),\C)$ is known to be the nilpotent cone 
$\mc{N}$ of $\g$ by \cite{GinKum93} 
and $\mc{N}=\overline{\mathbb{O}}_q$ by \cite{Ara09b}. For $q \le h$, 
the authors of \cite{BenNakPar14} proved that 
$\on{Spec}H^\bullet(\mf{u}_{\zeta}(\g),\C)$ 
is the closure $G.\mf{u}_{I_q}$  of some Richardson nilpotent orbit 
described in term of a set 
of simple roots $I_q \subset \Pi$ which only depends on $q$. 
Here $\Pi$ is a set of simple roots of $\g$, and $\mathfrak{u}_I$ denotes the nilpotent radical of the standard parabolic  subalgebra of $\g$ corresponding to a subset $I \subset \Pi$. It is straightforward to check that $G.\mf{u}_{I_q}$ coincides with 
$\overline{\mathbb{O}_{q}}$ from the tables in \cite{Ara09b}. 
The equality \eqref{eq:small_quantum} suggests a strong connection between 
the representations of $\mathbb{L}_k(\g)$ and of $\mf{u}_{\zeta}(\g)$,
which will be studied in a forthcoming paper \cite{ACK}.

Finally, we remark that
the statement of Theorem \ref{MainThm} for $(q,\check{r})=1$,
that is,
the equality
$\on{Var}(J_{\rho/q-\rho})
=\mc{N}_q$,
 seems to be true without the assumption that
$(q,\check{r})=1$.

\subsection*{Acknowledgements.} 
The authors thank Anna Lachowska for bringing
the paper \cite{BenNakPar14} to our attention.
TA is partially supported by JSPS KAKENHI Grant Numbers 17H01086, 17K18724,
21H04993.
JvE was supported by the Serrapilheira Institute
(grant number Serra -- 1912-31433) and by CNPq grants 409582/2016-0 and 303806/2017-6. 
AM is partially  supported by ANR Project GeoLie Grant number ANR-15-CE40-0012.

\section{Admissible affine vertex algebras} 
\label{sec:admissible}

In this section we prove the inclusions $\subset$ in Theorem \ref{MainThm}
using 
the representation theory of admissible affine vertex algebras.

Fix a triangular decomposition $\g=\n_-\+\h\+\n_+$,
$\Delta$ the set of roots of $\g$,
$\Delta_+$ a set of positive roots.
Let $Q$, $P$, $\check{Q}$, $\check{P}$
be the root lattice, the weight lattice, the coroot lattice and the coweight lattice, respectively. Let $\alpha^{\vee}=2\alpha/(\alpha|\alpha)$ denote the coroot corresponding to the root $\alpha$. The Weyl vector and covector are given, respectively, by $\rho = \tfrac{1}{2} \sum_{\alpha \in \Delta} \alpha$ and $\check{\rho} = \tfrac{1}{2} \sum_{\alpha \in \Delta} \alpha^\vee$. Let $h$ and $\check{h}$ be the 
Coxeter number 
and the dual Coxeter number of $\g$, respectively. We have
 \begin{align}
 h=( \check{\rho}|\theta )+1,\quad
\check{h}= ( \rho|\check{\theta})+1,
\label{eq:Coxternum}
\end{align}
where $\theta$ is the highest root of $\g$. We denote by $\check{\g}$ the Langlands dual Lie algebra of $\g$, characterized by its root system $\Delta(\check{\g}) = \{\alpha^\vee / \check{r} \colon \alpha \in \Delta(\g)\}$. 

Let $\widehat{\g} = \g[t,t^{-1}] \oplus \C K 
$  be the affine Kac-Moody algebra, 
with the commutation relations:  
$$[x t^m,y t^n] = [x,y] t^{m+n} + m \delta_{m+n,0} (x|y) K,  
\quad
[K,\affg]=0,$$
for all $x,y \in\g$ and all $m,n\in\Z$, where 
$(~|~)=\displaystyle{\frac{1}{2 \check{h}}}\times$Killing form 
is the normalized invariant inner product of $\g$ and 
and $x t^n$ stands for $x \otimes t^n$, for $x \in \g$, $n \in \Z$. 
Let $\widehat{\g}=\widehat{\mf{n}}_-\+ \widehat{\h}\+\widehat{\mf{n}}_+$
be the 
standard triangular decomposition,
that is,
$\widehat{\h} = \h \oplus \C K$  the  Cartan subalgebra of $\widehat{\g}$,
$\widehat{\mf{n}}_+=\mf{n}_++t\g[t]$,
$\widehat{\mf{n}}_-=\mf{n}_-+t^{-1}\g[t^{-1}]$.
Let $\affh^*=\h^*\+ \C\Lam_0$ be the dual of $\widehat{\h}$,
where  $(\Lam_0|K)=1$,
$(\Lam_0|\h)=0$. The affine Weyl vector is $\hat{\rho}=\rho+\check{h}\Lam_0$.

 Let $\tilde{\h}=\affh\+ \C D$ be the extended Cartan subalgebra of $\affg$,
 $\tilde{\h}^*=\affh^*\+ \C \delta$  the dual of $\tilde{\h}$,
with  $(\delta|D)=1$, $(\delta|\h)=0$.
 Let $\Deltahatre\subset \tilde{\h}^*$ be the real root system,
$\Deltahatre_+$ the standard set of positive roots in  $\Deltahatre$.
We have
\begin{align*}
 \Deltahatre &= \{\alpha+n\delta  \mid n \in \Z, \alpha \in \Delta\}= \Deltahatre_+\sqcup  \left(-\Deltahatre_+\right),\\
 &\Deltahatre_+=
  \Delta_+ 
 \sqcup \{\alpha+n\delta\mid \alpha\in \Delta 
 , \ n> 0\}.
\end{align*}
Let $\widehat{W}=W\ltimes \check{Q}$ be the affine Weyl group of $\affg$ and $\widetilde{W}=W\ltimes \check{P}$ the extended affine Weyl group. Here the action of the element $\beta \in \check{Q}$ or $\check{P}$ is via the translation $t_\beta$ defined as
\begin{align*}
t_\beta(\lambda) = \lambda + (\lambda | \delta) \beta - \left[(\lambda | \beta) + \tfrac{1}{2}|\beta|^2 (\lambda | \delta) \right] \delta.
\end{align*}

For $\lam\in \widehat\h^*$, let 
$$\widehat{\Delta}(\lam)=\{\alpha\in \Deltahatre\mid \bra \lam+\hat{\rho},\alpha^{\vee}\ket
\in \Z
\},$$
the set of roots integral with respect to $\lam$. A weight $\lam \in \widehat\h^*$ 
is said to be \emph{admissible} if 
\begin{enumerate}
\item $\lam$ is regular dominant, that is,
$\bra \lam+\hat{\rho},\alpha^{\vee}\ket >0$ for all $\alpha\in \widehat{\Delta}_+(\lam) =\widehat{\Delta}(\lam)\cap 
\Deltahatre_+$,
\item $ \mathbb{Q}\Deltahatre = \mathbb{Q}\widehat{\Delta}(\lam)$.
\end{enumerate}
We shall denote by $\Adm^k$ the set of admissible weights of level $k$. The irreducible highest weight representation $\L(\lam)$ is called 
admissible if $\lam$ is admissible.

Given any $k\in\C$, 
let
\begin{align*} 
\V^k(\fing) =
 U(\affg)\otimes_{U(\mf{g}[t]\oplus \C K)}\C_k, 
\end{align*}
where $\C_k$ is  the one-dimensional representation 
 of 
$\g[t] \oplus \C K$ on which $\g[t]$ 
acts by 0 and $K$ acts as a multiplication by the scalar $k$.
There is a unique vertex algebra structure 
on $\V^k(\g)$ such that $\vac$ is the image of $1\otimes 1$ 
in $\V^k(\g)$ and 
$$x(z) :=(x_{(-1)}\vac)(z)= \sum_{n\in\Z} (xt^n) z^{-n-1}$$
for all $x\in \g$, where we regard $\g$ as a subspace of $\V^k(\g)$ 
through the embedding $x \in \g \hookrightarrow x_{(-1)}\vac \in \V^k(\g)$. 
The vertex algebra $\V^k(\g)$ is called the {\em universal affine vertex algebra} associated with $\g$ at level $k$.

Any graded quotient  of $\V^k(\g)$ inherits a vertex algebra structure from 
$\V^k(\g)$.
In particular,
the unique simple graded quotient $\L_k(\g)$ of $\V^k(\g)$
is a vertex algebra, isomorphic to $\L(k\Lambda_0)$ as a $\widehat{\g}$-module.

For a graded quotient $V$ of $\V^k(\g)$,
the Zhu $C_2$-algebra \cite{Zhu96} can be defined as
the quotient
$$R_V=V/t^{-2}\g[t^{-1}] V.$$
There is a surjective linear map
$S(\g)\ra R_V$ that sends the monomial $x_1\dots x_r$, $x_i\in \g$,
to  $(x_1t^{-1})\dots (x_rt^{-1})|0\ket+t^{-2}\g[t^{-1}] V $,
and the kernel of this map is a Poisson ideal
of $S(\g)=\C[\g^*]$. 
Hence, $R_V$ is a Poisson algebra.
The associated variety of $V$ is by definition
the Poisson variety
\begin{align*}
X_V=\on{Specm}R_V,
\end{align*}
which is a $G$-invariant and 
conic subvariety of $\g^*$.


The simple vertex algebra $\L_k(\g)$ is called the {admissible affine vertex algebra}
if $\L_k(\g)$ is {admissible} as representation over the  affine Kac-Moody algebra
$\affg$ associated with $\g$,
or equivalently,
$k\Lam_0$ is the admissible weight.
If this is the case,
the level $k$ is called an {admissible number} for $\affg$.
By \cite[Proposition 1.2]{KacWak08}, $k$ is an admissible number if and only if 
\begin{align}
k+\check{h}=\frac{p}{q},\quad p,q\in \Z_{\geq 1},\ 
(p,q)=1,\quad  p\geq \begin{cases}
\check{h}&\text{if }(q,\check{r})=1,\\
h&\text{if }(q,\check{r})\ne 1.
\end{cases}
\label{eq:ad-number}
\end{align}
If this is the case we have
\begin{align*}
\widehat{\Delta}(k\Lam_0)=
\begin{cases}
\{\alpha+nq\delta\mid \alpha\in \Delta,\ n\in \Z\}
&\text{if }(q,\check{r})=1,\\
\{\alpha+nq\delta\mid \alpha\in \DeltaLong,\ n\in \Z\}\\
\quad \sqcup
\{\alpha+\frac{nq}{\check{r}}\delta\mid \alpha\in \DeltaShort,\ n\in \Z\}&\text{if }(q,\check{r})\ne 1,
\end{cases}
\end{align*}
where 
$\DeltaLong$ and $\DeltaShort$ are the sets of long roots and short roots, respectively.

We now recall some results on admissible affine vertex algebras, their associated varieties and modules. To state the results we recall the following subset of admissible weights
\begin{align*}
\Adm^k_{\circ}=\{\lam\in \Adm^k\mid \widehat{\Delta}(\lam)=w(\widehat{\Delta}(k\Lam_0))
\text{ for some }w\in \widetilde{W} 
\}.
\end{align*}

\begin{Th}[\cite{Ara09b}]
Let $k$ be an admissible number with  denominator $q$  as in \eqref{eq:ad-number}.
Then
\begin{align*}
X_{\L_k(\g)}=\overline{\O}_q.
\end{align*}
\end{Th}

\begin{Th}[\cite{A12-2}]\label{Th:classification-of-simple-modules}
Let $k$ be an admissible number,
$\lam\in \affh^*$. 
Then $\L(\lam)$ is a $\L_k(\g)$-module if and only if 
$\lam\in \Adm^k_{\circ}$.
\end{Th}

\begin{Pro}\label{Pro:lamq-is-adm}
For a non-negative integer $q$ 
define $\hat\lam_q\in \affh^*$ by
\begin{align*}
\hat\lam_q+\hat{\rho}
=\begin{cases}\frac{1}{q}\left(\rho+p\Lam_0\right)&\text{ if }(q,\check{r})=1,\\ 
\frac{1}{q}\left(\check{\rho}+p\Lam_0\right)&\text{ if }(q,\check{r})\ne 1.
\end{cases}
\end{align*}
If $k$ is an admissible number with denominator $q$, as in \eqref{eq:ad-number}, then $\hat\lam_q\in \Adm^k_{\circ}$.
\end{Pro}

\begin{proof}
By \eqref{eq:Coxternum},
we have
$\bra \hat\lam_q+\hat{\rho},\alpha^{\vee}\ket >0$
for $\alpha\in \Deltahatre_+$.
Hence $\hat\lam_q$ is regular dominant.
It is remain to show that
$\widehat{\Delta}(\hat\lam_q)=y(\widehat{\Delta}(k\Lam_0))$ for some $y\in \widetilde{W}$.
Firstly we suppose that $(q,\check{r})=1$,
so that
 $(\check{r} p,q)=1$. 
Take $c,d\in \Z$ such that 
$c\check{r}p+dq=-1$,
and set
$$\mu=c\check{r}\rho\in \check{P}.$$
Then we have 
$t_{-\mu}(\alpha+n\delta)=\alpha+\left[n + c\check{r}({\rho} | \alpha)\right]\delta$ and it follows that
\begin{align*}
\bra \hat\lam_q+\hat{\rho},t_{-\mu}(\alpha+n\delta)^{\vee}\ket=
\frac{1}{q}\left(\frac{2np}{(\alpha|\alpha)}+(1+c\check{r}p) (\rho | \alpha^{\vee})\right)\equiv \frac{2}{(\alpha|\alpha)}
\frac{np}{q}\pmod{\Z}.
\end{align*}
Thus
$t_{-\mu}(\alpha+n\delta)\in \widehat{\Delta}(\hat\lam_q)$ if and only if $n\equiv 0\pmod{q}$, and so we conclude that
\begin{align}
\widehat{\Delta}(\hat\lam_q)=t_{-\mu}(\widehat{\Delta}(k\Lam_0)).
\end{align}
Next we suppose that $(q,\check{r})>1$ so that $\check{r}|q$.
Take $c,d\in \Z$ such that $cp-dq=-1$,
and set
\begin{align*}
\mu=c\check{\rho}\in \check{P}.
\end{align*}
Then we have $t_{-\mu}(\alpha+n\delta)=\alpha+\left[n + c(\check{\rho} | \alpha)\right]\delta$ and it follows that
\begin{align*}
\bra \hat\lam_q+\hat{\rho},t_{-\mu}(\alpha+n\delta)^{\vee}\ket=
\frac{1}{q}\left(\frac{2np}{(\alpha|\alpha)}+(1+cp)(\check{\rho} | \alpha^{\vee})\right)\equiv \frac{2}{(\alpha|\alpha)}
\frac{np}{q}\pmod{\Z}.
\end{align*}
Thus,
$t_{-\mu}(\alpha+n\delta)\in \widehat{\Delta}(\hat\lam_q)$ if and only if $n\equiv \begin{cases}0\pmod{q}&\text{if }
\alpha\in \DeltaLong\\
0\pmod{q/\check{r}}&\text{if }\alpha\in \DeltaShort\end{cases}$.
We have thus shown that $\widehat{\Delta}(\hat\lam_q)=t_{-\mu}(\widehat{\Delta}(k\Lam_0))$.
\end{proof}

Let $\lam_q$ be the 
the restriction of a weight $\hat\lam_q \in \affh^*$ to $\h$. 
Then 
\begin{align}
\lam_q=\begin{cases}
\rho/q-\rho&\text{if }(q,\check{r})=1,\\
\check{\rho}/q-\rho&\text{if }(q,\check{r})\ne 1
\end{cases}
\end{align}
as in introduction.

\begin{Pro}\label{Pro:included}
We have 
$\on{Var}(J_{\lam_q})\subset \overline{\O}_q$.
\end{Pro}
\begin{proof}
Let $k$ be admissible number with denominator $q$,
and
let $\on{Zhu}(\L_k(\g))$ be the Zhu algebra of $\L_k(\g)$.
Then $\on{Zhu}(\L_k(\g))\cong U(\g)/I_k$ for some two-sided ideal $I_k$ of $U(\g)$.
By \cite[Theorem 9.5]{A2012Dec},
we have $$\on{Var}(I_k)=X_{\L_k(\g)}=\overline{\O}_q.$$
By Proposition \ref{Pro:lamq-is-adm} we have $\lam_q\in \Adm^k_{\circ}$, and so by Theorem \ref{Th:classification-of-simple-modules} in turn, 
$\L(\hat\lam_q)$ is an $\L_k(\g)$-module. Hence,
$L({\lam}_q)$ is a $\on{Zhu}(\L_k(\g))$-module.
This immediately implies $I_k\subset J_{ \lam_q}$,
and therefore 
$\on{Var}(J_{ \lam_q})\subset \on{Var}(I_k)=\overline{\O}_q$.

\end{proof}

\section{Proof of Theorem \ref{MainThm}}

In view of Proposition \ref{Pro:included},
it suffices to show that 
\begin{align}
\dim \on{Var}(J_{\lam_q})=\dim \overline{ \O}_q.
\end{align}
Since $\lam_q$ is a regular dominant weight, we have
\begin{align*}
\dim \on{Var}(J_{\lam_q})=\dim \mc{N}-|\Delta(\lam_q)|
\end{align*}
by \cite[Corollary 3.5]{Jos78},
where $$\Delta(\lam_q)=\{\alpha\in \Delta\mid \bra \lam_q+\rho,\alpha^{\vee}\ket \in \Z\}=\widehat{\Delta}(\hat\lam_q)\cap \Delta.$$
Therefore it is sufficient to show that 
\begin{align}
|\Delta(\lam_q)|=\dim \mc{N}-\dim \overline{\O}_q,
\label{eq:required-equality}
\end{align}
or rather (substituting an equivalent expression for the right hand side), that 
$$|\Delta(\lam_q)| = \dim \g^f - \on{rk}(\g),$$ 
for $f \in \O_q$, where $\g^f$ is the centralizer of $f$ in $\g$. 
The rest of the section is devoted to a case-by-case proof of \eqref{eq:required-equality}. 

%
%

\subsection{Proof of \eqref{eq:required-equality} for the simple classical Lie algebras}
\label{sub:classical}
In this paragraph, we show that \eqref{eq:required-equality} holds for all simple Lie 
algebra $\g$ of classical type. 
Let $n\in \Z_{>0}$ and assume that $\g$ is either $\sl_n$, $\so_n$ or $\sp_n$. 

First we set up notation. We denote by $\P(n)$ the set 
of partitions of $n$ and, unless otherwise specified, we write 
an element $\bs{\lambda}$ 
of $\P(n)$ 
as a decreasing sequence $\bs{\lambda}=(\lambda_1,\ldots,\lambda_r)$ of positive integers. 
Thus, 
$$
\lambda_1 \ge \cdots \ge \lambda_r \ge 1\quad 
\text{ and }\quad  
\lambda_1 + \cdots + \lambda_r  = n.
$$ 
We write $\bs{\lam}^t$ for the dual partition of $\bs{\lam}$.

$\bullet$ By \cite[Theorem~5.1.1]{CMa}, nilpotent orbits of $\sl_{n}$ are 
parametrized by $\P(n)$. For $\bs{\lambda}\in\P(n)$, 
we shall denote by 
$\O_{\bs{\lambda}}$ the corresponding nilpotent orbit of 
$\sl_n$. 
For $f \in \O_{\bs{\lambda}}$ the dimension of the centralizer of $f$ in $\g=\sl_{n}$ is given by
$$\dim \g^{f}= \sum_{i=1}^t \mu_i^2 - 1,$$
where $(\mu_1,\ldots,\mu_t) = \bs{\lam}^t$.

\noindent 
$\bullet$ Set 
$$\P_{1}(n):=\{\bs{\lambda} \in \P(n)\; \colon \; \text{multiplicity of each even 
part is even}\}.
$$ 
By \cite[Theorems 5.1.2 and 5.1.4]{CMa}, 
nilpotent orbits of $\so_{n}$ 
are parametrized by $\P_1(n)$, with the exception that each 
{\em very even} 
partition $\bs{\lambda} \in\P_{1}(n)$ (i.e., $\bs{\lambda}$ has only even parts) 
corresponds to two nilpotent orbits.
For $\bs{\lambda}\in \P_1(n)$, not very even, we shall denote by $\O_{\bs{\lambda}}$ 
the corresponding nilpotent orbit of $\so_n$. 
(Very even nilpotent orbits do not appear in this work.) 
For $f \in \O_{\bs{\lambda}}$ the dimension of the centralizer of $f$ in $\g=\so_{n}$ is given by
$$\dim \g^{f}= \frac{1}{2}\left(\sum_{i=1}^t \mu_i^2 -  
 \#\{i \; \colon \;  \text{$\lam_i$ is odd}\}\right),$$
where $(\mu_1,\ldots,\mu_t) = \bs{\lam}^t$.

\noindent 
$\bullet$ 
Set 
$$
\P_{-1}(n):=\{\bs{\lambda} \in \P(n)\; \colon \; \text{multiplicity of each odd part is even}\}.
$$ 
By \cite[Theorem~5.1.3]{CMa}, nilpotent orbits of $\sp_{n}$ 
are parametrized by $\P_{-1}(n)$. 
For $\bs{\lambda}=(\lambda_{1},\dots ,\lambda_{r})\in \P_{-1}(n)$, 
we shall denote by $\O_{\bs{\lambda}}$ 
the corresponding nilpotent orbit of $\sp_{n}$. 
For $f \in \O_{\bs{\lambda}}$ the dimension of the centralizer of $f$ in $\g=\sp_{n}$ is given by
$$\dim \g^{f}= \frac{1}{2}\left(\sum_{i=1}^t \mu_i^2 + \#\{i \; \colon \; \text{$\lam_i$ is odd}\}\right),$$
where $(\mu_1,\ldots,\mu_t) = \bs{\lam}^t$.

We recall that the \emph{height} of an element $\beta = \sum_{\alpha \in \Delta} k_\alpha \alpha$ of the root lattice $Q$ is by definition the integer $\on{ht}(\beta) = \sum_{\alpha \in \Delta} k_\alpha$. Equivalently $\on{ht}(\beta) = \bra \check{\rho}, \beta \ket$. In order to show that \eqref{eq:required-equality} holds in the classical cases, 
we exploit a  combinatorial  formula for
$\# \{\alpha \in \Delta \colon 
q | \on{ht}(\alpha) \in \Z \}$ proved in \cite[\S 3.2]{ElaKacVin08}. To state the formula we first write 
$$n=q m_0+s_0,$$ 
with $1 \le s_0 \le q$ if $\g=\so_n$ and $n$ even, and $0 \le s_0 \le q-1$ 
in all other cases. We consider the partition $(q^{m_0},s_0)$ of $n$, and we set 
$$
K_n(q) =m_0^2 (q-s_0)+ (m_0+1)^2 s_0.
$$
In fact if $\g = \sl_n$ then $K_n(q)-1$ is the dimension of the centralizer of a nilpotent 
element $f \in \g$ associated with the partition $(q^{m_0},s_0)$. By \cite[\S 3.2]{ElaKacVin08}, 
\begin{align}
\label{eq:EKV_relation}
\# \{\alpha \in \Delta \colon 
q | \on{ht}(\alpha) \in \Z \} = d_\g(q) - \on{rk}(\g),
\end{align}
where $d_\g(q)$ is expressed in terms of $n$ and $q$ as follows: 
\begin{enumerate}
\item for $\g=\sl_n$,
$$d_\g(q)=K_n(q)-1.$$
\item for $\g=\sp_n$,
$$2 d_\g(q)=K_n(q)+\begin{cases} m_0, & \text{ if } q \text{ is odd, }m_0  \text{ is even}, \\
m_0+1, & \text{ if } q,m_0  \text{ are odd}, \\
0, & \text{ if } q  \text{ is even}.\\
\end{cases}$$
\item for $\g=\so_n$, $n$ odd, 
$$2 d_\g(q)=K_n(q)- \begin{cases} m_0, & \text{ if } q,m_0 \text{ are odd},\\
m_0+1, & \text{ if } q \text{ is odd, }m_0  \text{ is even}, \\
2m_0+1, & \text{ if } q  \text{ is even}.\\
\end{cases}$$
\item for $\g=\so_n$, $n$ even, 
$$2 d_\g(q)=K_n(q)- \begin{cases} m_0, & \text{ if } q  \text{ is odd, }m_0  \text{ is even}, \\
m_0+1, &  \text{ if } q,m_0 \text{ are odd},\\
2m_0, &  \text{ if } q,m_0 \text{ are even},\\
2(m_0+1), & \text{ if } q  \text{ is even, }m_0 \text{ is odd}.\\
\end{cases}$$
\end{enumerate}

\subsubsection{First case: $(q,\check{r})=1$}
\label{subsub:princ}
In this case we have 
$$
|\Delta(\lam_q)|
= \# \{\alpha \in \Delta \colon \bra \rho,\alpha^{\vee}\ket \in q\,\Z \}.
%
$$
From \eqref{eq:EKV_relation}, using $\on{rk}(\g)= \on{rk}(\check{\g})$, we have 
\begin{align*}
%
 \# \{\alpha \in \Delta(\check{\g}) \colon \bra \rho^\vee,\alpha\ket \in q\,\Z \} 
&= \# \{\alpha \in \Delta(\check{\g}) \colon q | \on{ht}(\alpha) \} \\
&= d_{\check{\g}}(q) - \on{rk}(\g),
\end{align*}
(this holds whether or not $(q,\check{r})=1$). 
Using Lemma \ref{Lem:verif_d(q)_dual} below, we conclude that it is enough to verify the equality
\begin{align}
\label{eq:verification_d(q)}
\dim \g^f  = d_{\g}(q)
\end{align}
for $f \in \O_q$. 

\begin{lemma}
\label{Lem:verif_d(q)_dual} 
For any integer $q$, we have $d_{\check{\g}}(q)=d_{\g}(q)$. 
\end{lemma}

\begin{proof}
Notice that $\g$ and $\check{\g}$ 
share the same rank $\ell$ and the same exponents $m_1 \le m_2\le \cdots \le m_\ell$. 
But it is known that, denoting the number of positive roots of fixed height $i$ by $p_i$, 
the sequence $p_1 \ge p_2 \ge \cdots \ge p_{h-1}$, where $h$ is the Coxeter number of $\g$, 
gives a partition of $\Delta_+$ which is the dual partition to the partition $m_{\ell} \ge \cdots \ge m_2 \ge m_1$ 
given by the exponents (\cite[Theorem 9.1]{Steinberg}). 
Then the statement follows from the equality $d_{\check{\g}}(q) 
= \# \{\alpha \in \Delta(\check{\g}) \colon q | \on{ht}(\alpha) \}+ \on{rk}(\g)$.
\end{proof}

The rest of this paragraph is devoted to checking \eqref{eq:verification_d(q)} 
for each classical type.
 
For $\g=\sl_n$ (resp.~$\g=\so_n$, $\g=\sp_n$), we 
let $\bs{\lam}$ be the the element of $\mathscr{P}(n)$ (resp.~$\mathscr{P}_{1}(n)$, $\mathscr{P}_{-1}(n)$) 
associated with the nilpotent orbit $\O_q$, 
and we choose a nilpotent element $f \in \O_q$.

\subsubsection*{Case $\g=\sl_n$}
In this case $\g=\check{\g}$ 
and by \cite[Table 2]{Ara09b}, $\bs{\lam} = (q^m,s)$, so \eqref{eq:verification_d(q)} 
obviously holds.  

\subsubsection*{Case $\g=\so_n$, $n$ even}  
In this case $\g=\check{\g}$. Following \cite[Table 2]{Ara09b} there are four sub-cases to analyse.
\begin{enumerate}
\item $\bs{\lam} = (q^m,s)$ with $q,m$ odd, $0 \le s\le q$ odd. 
Then, clearly, 
$$\dim \g^f=\frac{1}{2}(m^2(q-s) +(m+1)^2 s - m-1) = \dim d_\g(q),$$
whence \eqref{eq:verification_d(q)}. 

\item $\bs{\lam} = (q^m,s,1)$ with $q$ odd, $m$ even, $0 \le s\le q-1$ odd. 
Then 
$$\dim \g^{f} = \frac{1}{2}(m^2(q-s)+(m+1)^2 (s-1)+(m+2)^2 -m-2).$$
On the other hand, since $n=m q +s+1$ with $0 \le s+1 \le q$, 
$$\d_{\g}(q)=\frac{1}{2}(m^2 (q -s-1) +(m+1)^2 (s+1) - m), $$ 
whence \eqref{eq:verification_d(q)}. 

\item $\bs{\lam} = (q+1,q^m,s)$ with $q,m$ even, $0 \le s\le q-1$ odd. 
Then 
$$\dim \g^{f} = \frac{1}{2}(1+ (m+1)^2(q-s)+(m+2)^2 s -2).$$
On the other hand, since $n=(m+1)q +s+1$ with $0 \le s+1 \le q$, 
$$\d_{\g}(q)=\frac{1}{2}((m+1)^2 (q -s-1) +(m+2)^2 (s+1) - 2(m+2)), $$ 
whence \eqref{eq:verification_d(q)}. 

\item $\bs{\lam} = (q+1,q^m,q-1,s,1)$ with $q,m$ even, $0 \le s\le q-1$ odd. 
Then 
$$\dim \g^{f} = \frac{1}{2}(1+ (m+1)^2 + (m+2)^2 (q-1-s)+(m+3)^2 (s-1) +(m+4)^2 -4).$$
On the other hand, since $n=(m+2)q +s+1$ with $0 \ge s+1 \ge q$, 
$$\d_{\g}(q)=\frac{1}{2}((m+2)^2 (q -s-1) +(m+3)^2 (s+1) - 2(m+2)), $$ 
whence \eqref{eq:verification_d(q)}. 
\end{enumerate}

\subsubsection*{Case $\g=\so_n$, $n$ odd}
Following \cite[Table 2]{Ara09b}, there are two sub-cases to analyse. 
\begin{enumerate}
\item $\bs{\lam} = (q^m,s)$ with $q$ odd, $m$ even, $0 \le s\le q$ odd.
Then 
$$\dim \g^{f} = \frac{1}{2}(m^2(q-s)+(m+1)^2 s -m-1) = \d_{\g}(q),$$
whence \eqref{eq:verification_d(q)}. 

\item $\bs{\lam} = (q^m,s,1)$ with $q,m$ odd, $0 \le s\le q-1$ odd. 
Then 
$$\dim \g^{f} = \frac{1}{2}(m^2(q-s)+(m+1)^2 (s-1)+(m+2)^2  -m-2).$$
On the other hand, $n-1=qm+s$ with $0 \le s \le q-1$. 
So
$$\d_{\check{\g}}(q)=\frac{1}{2}(m^2 (q -s) +(m+1)^2 s +m+1) = d_\g(q) ,$$ 
by Lemma \ref{Lem:verif_d(q)_dual}, whence \eqref{eq:verification_d(q)}. 
\end{enumerate}

\subsubsection*{Case $\g=\sp_n$}
Following \cite[Table 2]{Ara09b}, there are two sub-cases to analyse. 
\begin{enumerate}
\item $\bs{\lam} = (q^m,s)$ with $q$ odd, $m$ even, $0 \le s\le q-1$ even.
Then 
$$\dim \g^{f} = \frac{1}{2}(m^2(q-s)+(m+1)^2 s +m)=\d_{\g}(q),$$
whence \eqref{eq:verification_d(q)}. 

\item $\bs{\lam} = (q^m,q-1,s)$ with $q$ odd, $m$ even, $0 \le s\le q-1$ even. 
Then 
$$\dim \g^{f} = \frac{1}{2}(m^2+(m+1)^2 (q-1-s)+(m+2)^2 s +m).$$
On the other hand, $n+1=q (m+1) +s$ with $0 \le s \le q-1$. 
So 
$$\d_{\check{\g}}(q)=\frac{1}{2}((m+1)^2 (q -s) +(m+2)^2 s)  = d_\g(q),$$ 
by Lemma \ref{Lem:verif_d(q)_dual}, whence \eqref{eq:verification_d(q)}. 
\end{enumerate}

\subsubsection{Second case: $(q,\check{r}) =\check{r}$}
\label{subsub:coprinc}
In this case $q$ is necessarily even, $\check{r}=2$ and either $\g=\so_n$ with $n$ odd or else $\g=\sp_n$.
Note that 
\begin{align}
\nonumber 
|\Delta(\lam_q)| & = \# \{\alpha \in \Delta \colon 
\bra \check{\rho},\alpha^{\vee}\ket \in q\,\Z \}&\\ 
\label{eq:Delta_coprincipal}
& =
 \# \{\alpha \in \DeltaLong \colon 
q | \on{ht}(\alpha)  \}+
\# \{\alpha \in \DeltaShort \colon 
(q/2) | \on{ht}(\alpha)  \}.&
\end{align}

\subsubsection*{Case $\g=\so_n$, $n$ odd}
Firstly we rewrite \eqref{eq:Delta_coprincipal} as
\begin{align*}
 |\Delta(\lam_q)|  & = \# \{\alpha \in \Delta \colon 
q | \on{ht}(\alpha) \} & \\
& \qquad + \# \{\alpha \in \Delta_{\text{short}} \colon 
(q/2) | \on{ht}(\alpha)  \}
- \# \{\alpha \in \Delta_{\text{short}} \colon 
q | \on{ht}(\alpha) \}.&
\end{align*}
There are exactly $\on{rk}(\g) = (n-1)/2$ positive short roots, one of height $i$ for each integer $i=1,2,\ldots,(n-1)/2$. It follows easily that if we write $n=q m_0 +s_0$ where $0 \le s_0 \le q-1$, then we have
\begin{align*}
& \# \{\alpha \in \Delta_{\text{short}} \colon 
(q/2) | \on{ht}(\alpha)  \}
- \# \{\alpha \in \Delta_{\text{short}} \colon 
q | \on{ht}(\alpha)  \} 
= \begin{cases} 
m_0  & \text{ if } m_0 \text{ is even}, \\
m_0 +1 & \text{ if } m_0 \text{ is odd}.
\end{cases}
\end{align*}
We combine equation \eqref{eq:EKV_relation} with the preceding observations to deduce that, in the present case, \eqref{eq:required-equality} is equivalent to
\begin{align}
\label{eq:verification_d(q)-coprincipal-so_n}
\dim \g^f = d_\g(q) +\begin{cases} 
m_0 & \text{ if } m_0 \text{ is even}, \\
m_0 +1 & \text{ if } m_0 \text{ is odd}.
\end{cases}
\end{align}
We now verify that \eqref{eq:verification_d(q)-coprincipal-so_n} holds. 
Following \cite[Table 3]{Ara09b}, there are two sub-cases to analyse.

\begin{enumerate} 
\item $\bs{\lam}=(q^m,s)$, with $q,m$ even, $0 \le s \le q-1$ odd.   
In this case
$$\dim \g^f = \frac{1}{2}(m^2 (q-s)+(m+1)^2 s-1) $$
On the other hand $m_0=m$ is even, and so
$$d_\g(q)+m_0 = \frac{1}{2}(m^2 (q-s)+(m+1)^2 s-2m-1) + m $$
whence \eqref{eq:verification_d(q)-coprincipal-so_n}.

\item   $\bs{\lam}=(q^m,q-1,s,1)$, with $q,m$ even, $0 \le s\le q-1$ odd.   
In this case
$$\dim \g^f = \frac{1}{2}(m^2 +(m+1)^2(q-1-s)+(m+2)^2 (s-1)+(m+3)^2-3) $$
On the other hand $m_0=m+1$ is odd, and so
$$d_\g(q)+m_0 + 1 = \frac{1}{2}((m+1)^2 (q-s)+(m+2)^2 s-2(m+1)-1) +m+2$$
whence \eqref{eq:verification_d(q)-coprincipal-so_n}.
\end{enumerate}

\subsubsection*{Case $\g=\sp_n$}
There are exactly $\on{rk}(\g)=n/2$ positive long roots, one of height $2i+1$ for each integer $i=0,1,\ldots,n/2-1$. In particular
$$ \# \{\alpha \in \DeltaLong \colon 
q | \on{ht}(\alpha)  \}=0$$
since $q$ is even. 
Hence \eqref{eq:Delta_coprincipal} becomes
\begin{align*}
|\Delta(\lam_q)| 
& = \# \{\alpha \in \Delta \colon 
(q/2) | \on{ht}(\alpha)  \} 
- \# \{\alpha \in \DeltaLong \colon 
(q/2) | \on{ht}(\alpha)  \}. & 
\end{align*}
If $q/2$ is even, then by the above remarks the set $\{\alpha \in \DeltaLong \colon 
(q/2) | \on{ht}(\alpha)  \}$ is empty. If $q/2$ is odd, and we write $n = \left(q/2\right)m_1 +s_1$ where $0 \le s_1 \le \frac{q}{2}-1$, then it follows easily that
$$\{\alpha \in \DeltaLong \colon 
(q/2) | \on{ht}(\alpha)  \} = \begin{cases} 
m_1 & \text{ if } m_1 \text{ is  even,} \\
m_1+1 & \text{ if } m_1 \text{ is  odd.} \\
\end{cases}$$

By combining these facts with \eqref{eq:EKV_relation} we deduce that \eqref{eq:required-equality} is equivalent, in the present case, to
\begin{align}
\label{eq:verification_d(q)-coprincipal-sp_n}
\dim \g^f = d_\g(q/2) - \begin{cases} 
0 & \text{ if } q/2 \text{ is  even,} \\
m_1& \text{ if } q/2 \text{ is odd and } m_1 \text{ is even }, \\
m_1+1 &\text{ if } q/2 \text{ is odd and }  m_1 \text{ is odd}. \\
\end{cases}
\end{align}
We now verify that \eqref{eq:verification_d(q)-coprincipal-sp_n} holds. 
Following \cite[Table 3]{Ara09b}, there are three sub-cases to analyse. 

\begin{enumerate} 
\item $\bs{\lam}=((q/2)^m,s)$, with $q/2,m$ even, $0 \le s\le q/2-1$ even.   
In this case 
$$\dim \g^f = \frac{1}{2}(m^2(q/2-s)+(m+1)^2 s) =d_\g(q/2)-0,$$
whence \eqref{eq:verification_d(q)-coprincipal-sp_n}.

\item   $\bs{\lam}=(q/2+1,(q/2)^m,s)$, with $q/2$ odd, $m$ even, $0 \le s\le q/2-1$ even.   
In this case 
$$\dim \g^f = \frac{1}{2}(1+(m+1)^2(q/2-s)+(m+2)^2 s+m).$$
On the other hand, 
$n=(m+1)(q/2) +s+1$, with $m_1=m+1$ odd, and $0 \le s+1 \le q/2$. 

\begin{itemize} 
\item If $s< q/2-1$, then $0 \le s+1 <q/2$, and 
\begin{align*}
d_\g(q/2) - m_1& =\frac{1}{2}((m+1)^2(q/2-s-1)+(m+2)^2(s+1)+m+2) - m-2, &
\end{align*} 
whence \eqref{eq:verification_d(q)-coprincipal-sp_n}. 

\item If $s=q/2-1$, then $n=(m+2)(q/2)$ with $m_1=m+2$ even,  
and 
\begin{align*}
d_\g(q/2) - m_1  =\frac{1}{2}((m+1)^2(q/2)+(m+2)^2+m+2) - m-2 ,
\end{align*} 
whence \eqref{eq:verification_d(q)-coprincipal-sp_n}. 
\end{itemize}

\item  $\bs{\lam}=(q/2+1,(q/2)^m,q/2-1,s)$, with $q/2$ odd, $m$ even, $0 \le s\le q/2-1$ even.  
In this case 
$$\dim \g^f = \frac{1}{2}(1+(m+1)^2 +(m+2)^2(q/2-1-s)+(m+2)^2 s+m).$$
On the other hand, 
$n=(m+2)(q/2) + s$, with $m_1=m+2$ even,  $0 \le s \le q/2-1$, 
and 
\begin{align*}
d_\g(q/2) - m_1& =\frac{1}{2}((m+2)^2 (q/2-s)+(m+3)^2(s+1)-m-2) - (m+2) ,&
\end{align*} 
whence \eqref{eq:verification_d(q)-coprincipal-sp_n}. 
\end{enumerate}

\subsection{Proof of \eqref{eq:required-equality} for the simple exceptional Lie algebras}
\label{sub:exceptional}
In this paragraph, we show that \eqref{eq:required-equality} holds for all simple Lie 
algebra $\g$ of exceptional type. 


As in \S\ref{subsub:princ} above, when $(q,\check{r})=1$ we have 
using Lemma \ref{Lem:verif_d(q)_dual}, 
\begin{align*}
|\Delta(\lam_q)|
&= \# \{\alpha \in \Delta \colon 
\bra \rho,\alpha^{\vee}\ket \in q \, \Z \} \\
&=
 \# \{\alpha \in \Delta \colon 
q | \on{ht}(\alpha^\vee)  \} = d_{\check{\g}}(q) -\on{rk}(\g) 
= d_{\g}(q) -\on{rk}(\g),
\end{align*}
and hence, by \eqref{eq:EKV_relation}, 
it is enough to verify that the equality 
\eqref{eq:verification_d(q)} 
holds for any $f \in \O_q$. 
Note that the last three equalities hold even if $(q,\check{r})\not=1$. 
Using Tables \ref{Tab:G2-princ}, \ref{Tab:F4-princ}, \ref{Tab:E6-princ}, \ref{Tab:E7-princ} and \ref{Tab:E8-princ} below, 
we easily verify that \eqref{eq:verification_d(q)} holds 
for $\g$ any of the simple Lie algebras of type $G_2$, $F_4$, 
$E_6$, $E_7$, $E_8$, and for $(q,\check{r})=1$.  
In Tables \ref{Tab:G2-princ}, \ref{Tab:F4-princ}, we put in parentheses 
the values of $q$ which are not coprime with $\check{r}$. 
As we can see, the equality \eqref{eq:verification_d(q)} still holds in these cases.

As in \S\ref{subsub:coprinc} above, when $(q,\check{r})=\check{r}$ we have 
\begin{align*}
|\Delta(\lam_q)|
&=|\Delta(\check{\rho}/{q}-\rho)| \\
&=
 \# \{\alpha \in \DeltaLong \colon 
q | \on{ht}(\alpha) \}+
\# \{\alpha \in \DeltaShort \colon (q/\check{r}) | \on{ht}(\alpha)  \}. 
\end{align*}
Using this, we compute $|\Delta(\lam_q)|$ 
for the simple Lie algebras of type $G_2$ and $F_4$ with $(q,\check{r})=\check{r}$. 
From Tables \ref{Tab:G2-coprinc} and \ref{Tab:F4-coprinc}, 
we conclude that \eqref{eq:verification_d(q)} holds.


\bigskip

\begin{minipage}{.5\textwidth}
\centering
{\tiny
\begin{tabular}{|cccc|}
\hline
$q$ & $\mc{N}_{q}$ & $\dim \mc{N}_{q}$ & $|\Delta(\lam_q)|$ \\
\hline
$2$ & $\tilde{A}_1$ & $8$&  $4$ \\
$(3),4,5$ & $G_2(a_1)$ & $10$ & $2$ \\
$(6), \ge 7$ & $G_2$ & $12$ &$0$ \\
\hline
\end{tabular}
\captionof{table}{Data for $G_2$, $(q,\check{r})=1$}
\label{Tab:G2-princ}
}
\end{minipage}
\begin{minipage}{.5\textwidth}
\centering
{\tiny
\begin{tabular}{|cccc|}
\hline
$q$ & ${}^L\mc{N}_{q/\check{r}}$ & $\dim {}^L\mc{N}_{q/\check{r}}$ & $|\Delta(\lam_q)|$ \\
\hline
$3$ & ${A}_1$ & $6$& $6$ \\
$6, 9$ & $G_2(a_1)$ & $10$ & $2$ \\
$\ge 12$ & $G_2$ & $12$ & $0$ \\
\hline
\end{tabular}
\captionof{table}{Data for $G_2$, $(q,\check{r})=\check{r}$}
\label{Tab:G2-coprinc}
}
\end{minipage}

\bigskip


\begin{minipage}{.5\textwidth}
\centering
{\tiny
\begin{tabular}{|cccc|}
\hline
$q$ & $\mc{N}_{q}$ & $\dim \mc{N}_{q}$ & $|\Delta(\lam_q)|$ \\
\hline
$(2)$ & $A_1+\tilde{A}_1$ &$28$&  $20$ \\
$3$ & $\tilde{A}_2+{A}_1$ & $36$ & $12$  \\
$(4), 5$ & $F_4(a_3)$ &$40$& $8$  \\
$(6), 7$& $F_4(a_2)$&$44$& $4$  \\
$(8), 9, (10),11$& $F_4(a_1)$& $46$& $2$  \\
$(12), \ge 13$ &$F_4$&$48$&$0$\\
\hline
\end{tabular}
\captionof{table}{Data for $F_4$, $(q,\check{r})=1$}
\label{Tab:F4-princ}
}
\end{minipage}
\begin{minipage}{.5\textwidth}
\centering
{\tiny
\begin{tabular}{|cccc|}
\hline
$q$ & ${}^L\mc{N}_{q/\check{r}}$ & $\dim {}^L\mc{N}_{q/\check{r}}$ & $|\Delta(\lam_q)|$ \\
\hline
$2$ & $A_1$ &$16$&  $32$ \\
$4$ & $A_2+\tilde{A}_1$ & $34$ & $14$  \\
$6$ & $F_4(a_3)$ &$40$& $8$ \\
$8$& $B_3$&$42$& $6$  \\
$10$& $F_4(a_2)$& $44$& $4$ \\
$12,14,16 $&$F_4(a_1)$&$46$& $2$ \\
$\ge 18$ &$F_4$ &$48$ &$0$ \\
\hline
\end{tabular}
\captionof{table}{Data for $F_4$, $(q,\check{r})=\check{r}$}
\label{Tab:F4-coprinc}
}
\end{minipage}



\bigskip

\begin{minipage}{.5\textwidth}
\centering
{\tiny
\begin{tabular}{|cccc|}
\hline
$q$ & $\mc{N}_{q}$ & $\dim \mc{N}_{q}$ & $|\Delta(\lam_q)|$ \\
\hline
$2$ & $3A_1$ & $40$&  $32$ \\ 
$3$ & $2A_2+A_1$ & $54$ & $18$  \\
$4$  &$D_4(a_1)$ & $58$& $14$  \\
$5$&$A_4+A_1$ & $62$& $10$  \\
$6,7$&$E_6(a_3)$  & $66$ & $6$  \\
$8$&$D_5$& $68$& $4$  \\
$9,10,11$&$E_6(a_1)$&$70$&$2$\\ 
$\ge 12$ & $E_6$ & $72$&$0$\\
\hline
\end{tabular}
\captionof{table}{Data for $E_6$}
\label{Tab:E6-princ}
}
\end{minipage}
\begin{minipage}{.5\textwidth}
\centering
{\tiny
\begin{tabular}{|cccc|}
\hline
$q$ & $\mc{N}_{q}$ & $\dim \mc{N}_{q}$ & $|\Delta(\lambda_q)|$ \\
\hline
$2$ & $4A_1$ & $70$&  $56$ \\ 
$3$ & $2A_2+A_1$ & $90$ & $36$  \\
$4$  &$A_3+A_2+A_1$ & $100$& $26$  \\
$5$&$A_4+A_2$ & $106$& $20$  \\
$6$&$E_7(a_5)$  & $112$ & $14$ \\
$7$&$A_6$&$114$& $12$  \\
$8$&$E_7(a_4)$&$116$&  $10$  \\
$9$&$E_6(a_1)$ &$118$ & $8$  \\
$10,11$&$E_7(a_3)$&$120$& $6$  \\
$12,13$&$E_7(a_2)$&$122$& $4$  \\
$14,\ldots,17$&$E_7(a_1)$&$124$& $2$  \\
$\geq 18$&$E_7$& $126$ & $0$\\
\hline
\end{tabular}
\captionof{table}{Data for $E_7$}
\label{Tab:E7-princ}
}
\end{minipage}

\bigskip

\begin{minipage}{.5\textwidth}
\centering
{\tiny
\begin{tabular}{|cccc|}
\hline
$q$ & $\mc{N}_{q}$ & $\dim \mc{N}_{q}$ & $|\Delta(\lambda_q)|$ \\
\hline
$2$ & $4A_1$ & $128$& $112$ \\
$3$ & $2A_2+A_1$ & $168$ & $72$  \\
$4$ & $2A_3$ & $188$& $52$  \\
$5$ & $A_4+A_3$ & $200$& $40$  \\
$6$ & $E_8(a_7)$  & $208$ &$32$  \\
$7$ & $A_6+A_1$& $212$& $28$  \\
$8$ & $A_7$& $218$& $22$  \\
$9$ & $E_8(b_6)$ & $220$ & $20$  \\
$10,11$ & $E_8(a_6)$ & $224$ &$16$ \\
$12,13$ & $E_8(a_5)$ &  $228$ & $12$ \\ 
$14$  &$E_8(b_4)$ & $230$ & $10$  \\
$15,16,17$&$E_8(a_4)$ & $232$ & $8$  \\
$18,19$&$E_8(a_3)$  & $234$ & $6$  \\
$20,\ldots,23$&$E_8(a_2)$& $236$ & $4$  \\
$24,\ldots,29$&$E_8(a_1)$& $238$ & $2$  \\
$\ge 30$ &$E_8$&$240$&$0$\\
\hline
\end{tabular}
\captionof{table}{Data for $E_8$}
\label{Tab:E8-princ}
}
\end{minipage}

\newcommand{\etalchar}[1]{$^{#1}$}

\bibliographystyle{alpha}
\bibliography{/Users/tomoyuki/Documents/Dropbox/bib/math}

\end{document}